\documentclass[table,x11names]{amsart}
 
\newif\ifshort
\shortfalse

\ifshort
\usepackage[top=1.1in, bottom=1.1in, left=1.3in, right=1.3in]{geometry}
\fi

\usepackage{va, styell}

\usepackage{bm}
\usepackage{float}
\usepackage{tikz}
\usepackage{tikz-cd}

\usepackage[pdftex]{hyperref}
\hypersetup{colorlinks,citecolor=blue,linktocpage,hyperindex=true,backref=true}
\allowdisplaybreaks

\usepackage{etoolbox}

\makeatletter
\pretocmd{\section}{\addtocontents{toc}{\protect\addvspace{5\p@}}}{}{}
\makeatother

\usepackage{todonotes}

\date{August 12, 2022}

\title{Periods of elliptic surfaces with $p_g=q=1$}
\author{Philip Engel}
\email{philip.engel@uga.edu}
\address{Department of Mathematics, University of Georgia, Athens GA
  30602, USA}
  \author{Fran\c{c}ois Greer}
\email{greerfra@msu.edu}
\address{Department of Mathematics, Michigan State University, East Lansing MI 48824, USA}
  \author{Abigail Ward}
\email{arw204@cam.ac.uk}
\address{Centre for Mathematical Sciences, Cambridge University, Cambridge CB3 0WA, UK}

\begin{document}

\begin{abstract}
We prove that the period mapping is dominant for elliptic surfaces
over an elliptic curve with $12$ nodal fibers, and that its
degree is larger than $1$. This settles the final case of infinitesimal Torelli for a generic elliptic surface.
 \end{abstract}

\maketitle



\section{Introduction}
\label{sec:introduction}

In order to distinguish smooth projective varieties varying in a family with continuous parameters, it is often useful to integrate the holomorphic forms over topological cycles. This idea was used to great effect classically to distinguish smooth curves of a given genus $g>0$. A modern reformulation of this problem in higher dimension asks whether the period mapping from a moduli space of varieties to an associated space of {\it periods} is injective, either locally or globally on the source. We will show that while the local injectivity statement is true generically, the global statement fails for an important class of elliptic surfaces.

An {\it elliptic surface} is a smooth, projective surface $S$ equipped with a relatively 
minimal, genus one fibration $\pi\colon S\to C$ to a smooth curve and a distinguished  
section $s$. Moduli spaces $F_{g,d}$ of elliptic surfaces are indexed by two non-negative
integers, $g=g(C)$ and $d=\tfrac{1}{12}\chi_{\rm top}(S)$. Counted with 
multiplicity, there are $12d$ singular fibers. The canonical bundle of $S$ is pulled
back from a line bundle $L\otimes \omega_C$ of degree $d+2g-2$ on $C$.
We henceforth assume $d>0$ (that is, $S$ has at least one singular fiber) so that $p_g(S):=h^0(K_S)=g+d-1$.

In this paper, we focus on the moduli space $F:=F_{1,1}$.
Since $g(C)=1$, $K_S=\pi^*L$ for a degree $1$ line bundle
$L=\cO_C(p)$, and generically the fibration $\pi$ has $12$ singular fibers.
There is a morphism $S\to \oS$ contracting ADE configurations in
fibers not intersecting the section $s$. This contraction has a Weierstrass form \cite{kas1977weierstrass}
$$\oS = \{y^2 = x^3+ax+b\}\subset \mathbb{P}_C(L^2\oplus L^3\oplus \cO)$$ where 
$a\in H^0(C,L^4)$ and $b\in H^0(C,L^6)$. A quick parameter count
reveals that $\dim F = 1 + 4 + 6 - 1=10$ where the parameters are, respectively,
the $j$-invariant of $(C,p)$, the section $a$, the section $b$, and the quotient by the
action of $\lambda\in\C^*$ via $(a,b)\mapsto (\lambda^4a,\lambda^6b)$.

Noether's formula implies that the Hodge numbers of $S$
are $h^{2,0}(S)=h^{1,0}(S)=1$ and $h^{1,1}(S)=12$.
The Neron-Severi group $\NS(S)=H^{1,1}(S,\C)\cap H^2(S,\Z)$ always contains the
classes of the fiber $f$ and section $s$ which have intersection numbers $s^2=-1$, $s\cdot f=1$, $f^2=0$.
Hence, there is a copy of the odd unimodular lattice $$I_{1,1}\simeq \Z s\oplus \Z(s+f)
\subset \NS(S).$$ Its orthogonal complement $\{s,f\}^\perp\subset H^2(S,\Z)$ is 
an even (since $[K_S]=f$), unimodular lattice of signature $(2,10)$, so it is isometric to $I\!I_{2,10}=H\oplus H\oplus E_8$.

Let $\Gamma:=O(I\!I_{2,10})$ and define the {\it period domain} to be $$\bD:=\mathbb{P}
\{x\in I\!I_{2,10}\otimes \C\,\big{|}\, x\cdot x=0,\, x\cdot \bar x>0\}.$$ It is a ten-dimensional Type IV 
Hermitian symmetric domain. By general results of Griffiths \cite{griffiths1968periods}, there is a
holomorphic {\it period map} $P\colon F\to \bD/\Gamma$ sending $[S]\in F$ to the line 
$H^{2,0}(S)\subset \{s,f\}^\perp\otimes \C$. This map is only well-defined mod $\Gamma$ 
since the isometry $\{s,f\}^\perp\to I\!I_{2,10}$ is ambiguous up to post-composition by an 
element of $\Gamma$. We may now state the first theorem of the paper:

\begin{theorem}\label{thm:main} $P$ is dominant. \end{theorem}

\begin{remark}
For surfaces $S$ with $h^{2,0}(S)\geq 2$, the associated period map {\it cannot} be
dominant due to Griffiths' transversality. The general member $S\in F_{g,d}$
satisfies $h^{2,0}(S)= 1$ only when $(g,d)=(1,1)$ or $(g,d)=(0,2)$.
In the latter case, the surfaces under consideration are elliptic K3 surfaces. By the Torelli
theorem for K3 surfaces \cite{piateski-shapiro1971torelli, looijenga1980torelli},
the period mapping gives an isomorphism onto the corresponding period space.
\end{remark}

A {\it local}, respectively ~{\it infinitesimal}, Torelli theorem verifies the local injectivity of $P$,
respectively injectivity of $dP$, at some point. Such a result implies
that $P$ is generically finite onto its image. A {\it generic} Torelli theorem
further proves that $P$ is generically one-to-one onto its image.
Finally, a {\it global} Torelli theorem implies that
$P$ is an embedding, or an isomorphism if the dimensions are appropriate.
We prove that, unlike for K3 surfaces,

\begin{theorem}\label{thm:main2} $\deg P>1$. Thus, generic Torelli is false
for $P\colon F\to \bD/\Gamma$. \end{theorem}

\begin{remark} By a result of L\"{o}nne \cite{lonne}, the monodromy representation for the universal family over $F$ is the subgroup of $O(I\!I_{2,10})$ preserving the connected component of $\bD$, so $P$ does not factor through $\bD /\Gamma'$ for any subgroup $\Gamma'\subset \Gamma$.\end{remark}\vspace{5pt}

To prove Theorem \ref{thm:main}, we employ a degeneration argument, similar to Friedman's
proof \cite{friedman1984a-new-proof} of the Torelli theorem for K3 surfaces. First
degenerate the base curve $C$ to a nodal curve $C_0$ formed from gluing two points on $\mathbb{P}^1$.
An elliptic fibration $S\to C$ may be degenerated to an elliptic fibration $S_0\to C_0$,
and the simplest case is when the fiber over the node of $C_0$ is smooth.
Normalizing, $$S_0^\nu=X\to \mathbb{P}^1=C_0^\nu$$ is an elliptic fibration with $(g,d)=(0,1)$, that is, a rational elliptic surface.
To reconstruct $S_0$ from $X$ we glue two smooth fibers $X_p$ and $X_q$ for $p,q\in \mathbb{P}^1$
in such a way that a section of $X\to \bP^1$ is glued to form a section of $S_0 \to C_0$.

The period map for such singular surfaces $S_0$ does not land in $\bD/\Gamma$, but
maps into the boundary divisor $\Delta$
of a toroidal extension $\bD/\Gamma\hookrightarrow (\bD/\Gamma)^{\rm II}$. 
It suffices to prove that the boundary period map
$P^{\rm II}\colon \{\textrm{moduli of }S_0\}\to \Delta$
is dominant. We find an explicit surface $S_0$ for which
any deformation of its period deforms its moduli. Thus $P^{\rm II}$
has at least one fiber containing a $0$-dimensional component,
implying dominance of $P^{\rm II}$, and in turn, $P$. \vspace{5pt}

To prove Theorem \ref{thm:main2}, we describe a second type of degeneration of $S\to C$,
to a fibration $S_0\to C$ (here the base stays constant) whose generic fiber is a nodal curve.
We analyze the limiting period mapping for these surfaces, and prove that they too map
dominantly into the boundary divisor $\Delta$. Since two different degenerations dominate the same divisor $\Delta$, we obtain that $\deg P>1$.

Our method of proof suggests an interesting conjecture. Each surface $S\in F$ contains two natural elliptic curves meeting at a point: the unique representative of the canonical class $K_S$ and the marked section curve $s$. The degenerations we employ in the proof leave one of these curves fixed and degenerate the other to a nodal curve. Conjecture \ref{fmconj} describes a birational involution of $F$, which commutes with the period mapping, and swaps the roles of the two natural elliptic curves.

\subsection*{History of the result}
In 1983, M.-H. Saito \cite{saito1983torelli} claimed to prove the following
infinitesimal Torelli theorem for elliptic surfaces: the differential $dP$ is injective
if the $j$-invariant map $j\colon C\to \mathbb{P}^1_j$ is non-constant, and $h^{2,0}(S) = g+d-1>0$.
However, in 2019, Ikeda \cite{ikeda2019bielliptic}
found a four-dimensional family $\cB \subset F_{1,1}$ for which
$P\big{|}_\cB$ has three-dimensional image, despite the general member of $\cB$ having
non-constant $j$-map. Thus, \cite{saito1983torelli} has a gap, but the proof still works when $\omega_S$
is basepoint free. Observe that $\omega_S \simeq \pi^*(L\otimes \omega_C)$ is basepoint free for all $S\in F_{g,d}$ when $g>0$ and $d>1$, and $\omega_S$ is basepoint free for generic $S\in F_{g,d}$ when $g>1$ and $d=1$. The only cases where $\omega_S$ fails to be basepoint free for generic $S$ are $(g,d)=(1,1)$ and $(g,d)=(0,1)$. The latter is the case of rational elliptic surfaces, where the period map is trivial.

In 2020, R. Kloosterman \cite{kloosterman2022infinitesimal} independently proved
that the infinitesimal Torelli theorem holds for elliptic surfaces with non-constant $j$-map when $d\neq 1$, or when $d=1$ and $h^0(C,L)=0$.
The techniques generalized those of Ki\u{\i} \cite{kii1978torelli} and Lieberman-Wilsker-Peters
\cite{lieberman1977torelli} from the $g=0$ case. Conversely, Kloosterman conjectured
\cite[Conj.~6.1]{kloosterman2022infinitesimal} that when
$d=h^0(C,L)=1$, the infinitesimal Torelli theorem is false.
But this condition holds at every point of $F_{1,1}$, so our Theorem \ref{thm:main} proves that Kloosterman's conjecture is, in fact, false. 

Regarding a generic Torelli theorem, Chakiris \cite{chakiris1982torelli} proved
that generic Torelli holds in the $g=0$, $d\geq 2$ case. Recently, Shepherd-Barron
\cite{shepherdbarron2020generic} has 
generalized these results to a higher genus base: elliptic surfaces $S\to C$ 
with $q=h^{1,0}(S)$ and $p_g=h^{2,0}(S)$ satisfying
the bounds $4p_g>5(q-1)$, $p_g\geq q+3$ also obey a generic Torelli theorem.
By our Theorem \ref{thm:main2}, generic Torelli is false when $p_g=q=1$.
Hence, the second linear inequality $p_g\geq q+3$ appears to be necessary for
Shepherd-Barron's results to hold.

\subsection*{Ackowledgements} The three authors were partially supported by NSF grants DMS-2201221, DMS-2302548, and DMS-2002183, respectively.

\section{Type II$_b$ degenerations}
\label{sec:period-map}

Let $\pi_0\colon S_0\to C_0$ be an elliptic fibration over an irreducible, nodal, arithmetic genus one curve $C_0$
with smooth fiber over the node, and $\chi_{\rm top}(S_0)=12$.
Such a fibration has a Weierstrass form $\{y^2=4x^3-a_0x-b_0\}$
with $a_0\in H^0(C_0,\mathcal{O}_{C_0}(4P_0))$ and $b_0\in H^0(C_0,\mathcal{O}_{C_0}(6P_0))$ 
for some point $P_0\in (C_0)_{\rm sm}$.
See Figure \ref{typeIIb}.

\begin{figure}
\includegraphics[width=2.5 in]{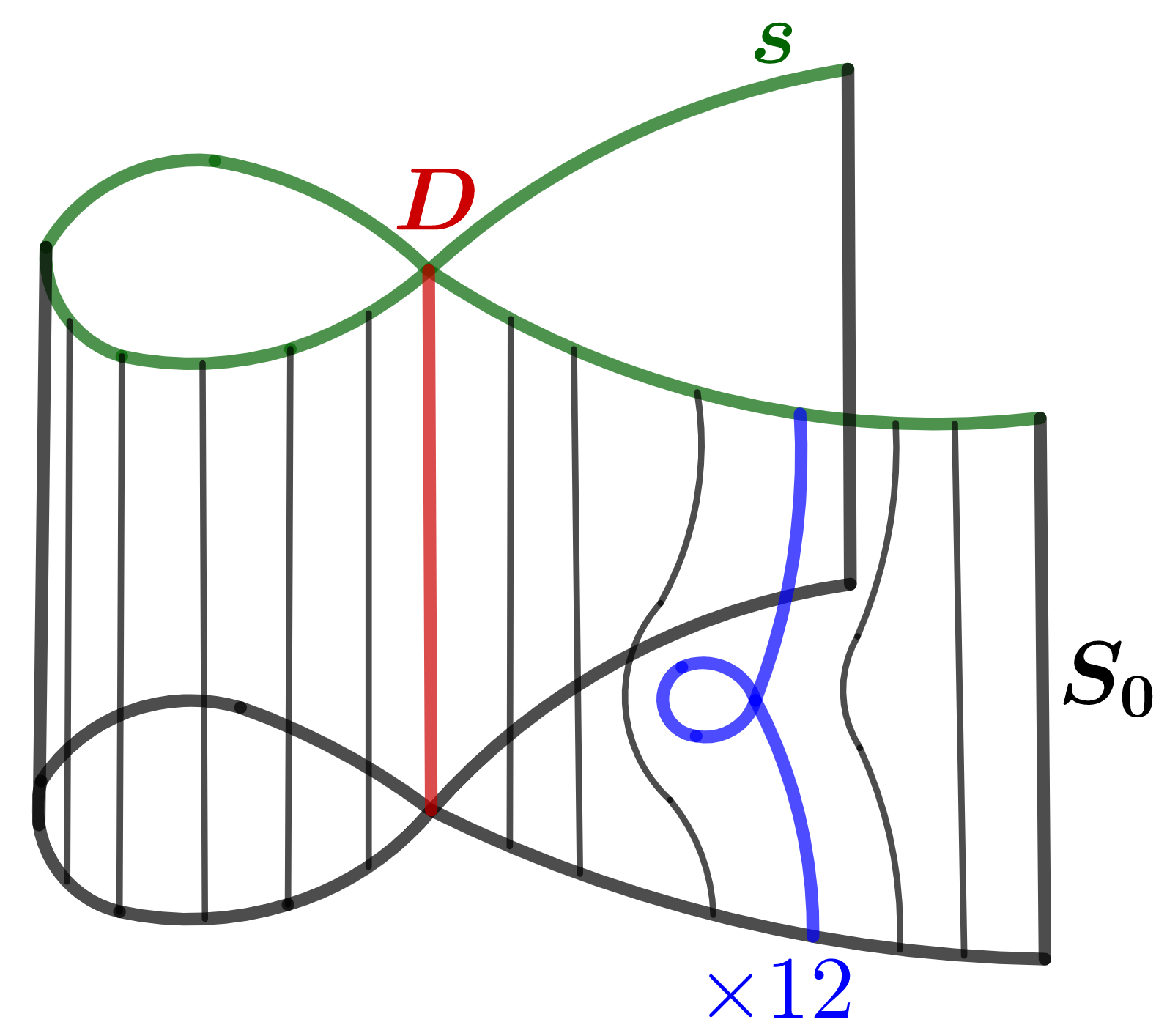}
\caption{A Type II$_b$ surface $S_0$
with double locus $D$ and section $s$.}
\label{typeIIb}
\end{figure}

 Let $C_0\hookrightarrow \cC$ be a smoothing over $(B,0)$ to a genus $1$ curve, with smooth total
space and let $\cP$ be an extension of $P_0$ to a section of $\rho\colon \cC\to (B,0)$.
Then, for any $k>0$, Cohomology and Base Change  \cite[III.12.11]{hartshorne1977algebraic-geometry}
 implies that $\rho_*\cO_{\cC}(k\cP)$
is a rank $k$ vector bundle over $B$. In particular, $a_0$, $b_0$ extend locally to sections
$a$, $b$ of $\rho_*\cO_\cC(4\cP)$, $\rho_*\cO_\cC(6\cP)$ and so we can smooth
the elliptic fibration $S_0\hookrightarrow \cS$ over $(B,0)$. The resulting total space $\cS$ is smooth
with $S_0$ reduced normal crossings. The double locus $D$ is the smooth elliptic curve
fibering over the node of $C_0$.

\begin{definition}\label{type-ii-def} We call such a degeneration $\cS\to \cC\to (B,0)$
a {\it Type II$_b$ degeneration} and we call the central fiber $S_0$ a
{\it Type II$_b$ elliptic surface}. \end{definition}

The subscript $b$ indicates that the base degenerates.
The terminology is motivated by a similar terminology in the classification
of one-parameter degenerations of K3 surfaces due to Kulikov and Persson-Pinkham
\cite{kulikov1977degenerations-of-k3-surfaces, persson1981degeneration-of-surfaces}. They
classify their $K_\cS$-trivial, reduced normal crossing degenerations into Types I, II, III
depending on the depth of the singularity stratification of $S_0$. Here, we instead have
$K_\cS=\cO_\cS(\cF)$ for a relative fiber $\cF\to (B,0)$.

As a reduced normal crossing degeneration, the Picard-Lefschetz transformation $T\colon H^2(S_t,\Z)\to
H^2(S_t,\Z)$ is unipotent and has a logarithm $N:=\log T$.
Furthermore, there is a formula for $N$ which can be deduced
from the Picard-Lefschetz transformation for a nodal degeneration of curves, or from
\cite[Thm.~5.6]{clemens1969picard}.

Let $\gamma_t\subset C_t$ denote the vanishing $1$-cycle of the node of $C_0$.
Since the fiber over the node of $C_0$ is smooth, the restriction of the elliptic fibration
$\pi_t\colon S_t\to C_t$ to the curve $\gamma_t$ is a topologically trivial $2$-torus bundle.
Trivialize it, and let $\alpha, \beta$ be oriented generators of the homology of some fiber.
Define $u:=[\gamma_t\times \alpha]\in H^2(S_t,\Z)$, $v:=[\gamma_t\times \beta]\in H^2(S_t,\Z)$.
Then:

\begin{proposition} $N(x) = (x\cdot u)v-(x\cdot v)u.$ \end{proposition}

Here $u,v\in \{s,f\}^\perp$ because $s,f$ are classes of line bundles
on the total space $\cS$, and hence monodromy-invariant. So the
classes $u,v$ determine a rank $2$ isotropic lattice
$I:=(\Z u\oplus \Z v)^{\rm sat}\subset I\!I_{2,10}$.

Let $U_I$ be the unipotent subgroup of ${\rm Stab}_\Gamma(I)$ acting trivially on
$I$ and $I^\perp/I$. From the theory of toroidal compactifications
\cite{ash1975smooth-compactifications} (see also
\cite[Sec.~1A]{looijenga2003compactifications-defined2}, 
\cite[Prop.~4.16]{alexeev2021compact} for the case of Type IV domains), the 
unipotent quotient $$\bD/U_I\hookrightarrow A_I$$ embeds as a
punctured disk bundle inside a $\C^*$-bundle $A_I\to I^\perp/I\otimes \cE$.
Here $\cE$ is the universal elliptic curve over $\C\setminus \R$
whose fiber over $\tau \in \C\setminus \R$
is the elliptic curve $\C/\Z\oplus \Z\tau$.
Since $T\in U_I$ the period map $P$ induces a holomorphic period map $B^*\to \bD/U_I$.

We enlarge $A_I\hookrightarrow \oA_I$ to a line bundle and define $(\bD/U_I)^{\rm II}$
as the closure of $\bD/U_I$ in $\oA_I$. This closure is a holomorphic disk bundle over $I^\perp/I\otimes \cE$.
The nilpotent orbit theorem \cite[Thm.~4.9]{schmid1973variation} (the case at hand follows
as in \cite[Thm.~4.2]{friedman1984a-new-proof})
implies that the period map from $B^*$ extends to a holomorphic map
$P\colon (B,0)\to (\bD/U_I)^{\rm II}$ sending $0$ into the boundary divisor $\Delta:=\oA_I\setminus A_I$. As the zero-section of the
line bundle, the boundary divisor is naturally isomorphic to $$\Delta\simeq I^\perp/I\otimes \cE.$$
Note that $I^\perp/I$ is an even, negative-definite, unimodular lattice of rank $8$, which
uniquely determines it to be $I^\perp/I=E_8$.

There is also a direct construction of the period point
$P(0)\in E_8\otimes \cE$ from the singular surface $S_0$ described as follows.
Let $X\to \mathbb{P}^1$ be the rational elliptic surface normalizing $S_0\to C_0$ and 
denote the section and fiber classes again by $s$ and $f$. Then
$\{s,f\}^\perp\subset H^2(X,\Z)$ is isomorphic to $E_8$. Let $X_p$ and
$X_q$ be the two elliptic fibers glued to form the double locus $D$ of $S_0$. A class $\gamma\in \{s,f\}^\perp$ defines a line bundle
$\cL_\gamma\in {\rm Pic}(X)$ and we declare \begin{equation}\label{first-period}
\psi_{S_0}(\gamma):=\cL_\gamma\big{|}_{X_p}
\otimes \cL_\gamma\big{|}_{X_q}^{-1}\in E:={\rm Pic}^0(X_p)\end{equation} where we have 
used the gluing isomorphism $X_p\to X_q$ to form the tensor product
of these two restrictions.

 Then $\psi_{S_0}$ defines a homomorphism
$\psi_{S_0}\in {\rm Hom}(E_8,E)\simeq E_8\otimes E$. Fixing an identification of $\{s,f\}^\perp$ with a fixed copy of the $E_8$ lattice, then deforming $S_0$ in moduli of Type II$_b$ surfaces,
we get a local holomorphic period map $$P^{\rm II}\colon {\rm Def}_{S_0}\to {\rm Hom}(E_8, \cE)$$ which is identical to the extension
of $P$ coming from the nilpotent orbit theorem. The equivalence of these two definitions of
the period map follows from Carlson's description \cite{carlson1985one} of the mixed Hodge
structure on $S_0$; see Section \ref{sec:mhs} and Proposition \ref{period-prop}.
From this description of the boundary period mapping, we see:

\begin{enumerate}
\item To prove that $P$ is dominant, it suffices to show that $P^{\rm II}$ is dominant from
the moduli of Type II$_b$ elliptic surfaces to ${\rm Hom}(E_8,\cE)$.\smallskip
\item On Type II$_b$ surfaces, the period map $P^{\rm II}$ is constructed
by comparing the restriction of a line bundle in
$\{s,f\}^\perp \subset {\rm Pic}(X)$ to the two glued fibers.
\end{enumerate}

Observe that (1) follows from the observation at the beginning
of this section that every Type II$_b$ elliptic surface is smoothable to the interior of $F$, so
the Zariski closure of ${\rm im}(P)\subset (\bD/\Gamma)^{\rm II}$ must contain
${\rm im}(P^{\rm II})$.

\section{Dominance of the period map}

Fix a smooth cubic $D\subset \mathbb{P}^2$ and let $\gamma\in PGL_3(\C)$
be generic. Then $D$ and $\gamma(D)$ generate a pencil of cubics with $9$ distinct
base points. Blowing up at the nine base points $D\cap \gamma(D)=\{p_1,\dots,p_9\}$
of this pencil, we get a rational elliptic surface $X\to \mathbb{P}^1,$ together with an
isomorphism $\gamma\colon D \to \gamma(D)$ between two of its fibers.
The nine blow-ups give rise to nine exceptional sections $F_1,\dots,F_9$
of the resulting elliptic fibration. Let $t\colon D\to D$ be an arbitrary translation and
consider the surface $S_0$ which results from gluing our two fibers of $X\to \mathbb{P}^1$
by the isomorphism $$\gamma \circ t \colon D\to \gamma(D).$$
This construction defines a family of singular surfaces $\cS\to U$
over a Zariski open subset $U\subset PGL_3(\C)\times E$ where $E:={\rm Pic}^0(D)$.

A very general surface over $(\gamma,t)$ does not have a section, as there
are only countably many sections of $X\to \mathbb{P}^1$; for a sufficiently
general translation $t$, none of these will glue to a section of the singular surface.
Still, for all such surfaces, there is a period homomorphism $\psi_{S_0}\colon H^2(X,\Z)\to E$
defined by (\ref{first-period}).
It descends to the rank $9$ quotient $L:=H^2(X,\Z)/\Z f$ because
$f|_D= \mathcal{O}_D$ and $f|_{\gamma(D)}= \mathcal{O}_{\gamma(D)}$. There is a translation
action of $t\in E$ on $U$ given by
$ (\gamma_0,\,t_0)\mapsto(\gamma_0,\,t_0\circ t)=:(\gamma_0',t_0')$.
It acts on the period homomorphism as follows: 
\begin{equation}\label{action}\psi_{S_0'}(v)=\psi_{S_0}(v)+(v\cdot f) t.\end{equation} From this formula, 
we deduce that the dominance
of the period map for Type II$_b$ elliptic surfaces
is equivalent to dominance of the more general period map \begin{align}\label{one}
PGL_3(\C)\times E\dashrightarrow {\rm Hom}(L,E).
\end{align} 

Consider the codimension one subtorus of ${\rm Hom}(L,E)$ for which $\psi_{S_0}(h)=0\in E$,
where $h$ is the pullback of the hyperplane class on $\mathbb{P}^2$. The inverse image of
this subtorus contains, as a component, the locus of $(\gamma,t)$ for which $t=0$,
because under a projective linear identification $\gamma$, we have
$\gamma^*\cO_{\gamma(D)}(1)=\cO_D(1)$. Thus, the dominance of 
(\ref{one}) is implied by the dominance of 
\begin{align}\label{two} PGL_3(\C)\dashrightarrow {\rm Hom}(H^2(X,\Z)/\Z f+\Z h,E).\end{align}
This follows because the action of $t\in E$ 
on ${\rm Hom}(L,E)$ described by (\ref{action}) is translation by an elliptic subcurve
transverse to the codimension $1$ subtorus of ${\rm Hom}(L,E)$ appearing
on the right-hand side of (\ref{two}).

Finally, $\Z^9 \simeq \textrm{span}\{F_i\,\big{|}\,i=1,\dots,9\}=h^\perp$
surjects onto $H^2(X,\Z)/\Z f+\Z h$.
Pulling back the period map to this lattice, we get a map
\begin{align} \begin{aligned}\label{three} PGL_3(\C)&\dashrightarrow {\rm Hom}(\Z^9,E)/\mathfrak S_9 \\ 
\gamma&\mapsto \{\psi_{S_0}(F_1),\dots,\psi_{S_0}(F_9)\}.\end{aligned} \end{align}
Here, the base points $D\cap \gamma(D)$, and hence the exceptional curves $F_i$,
are not canonically ordered; they are permuted by the monodromy of the universal family.
This is why we must quotient the target by the symmetric group
$\mathfrak S_9$. Since $\sum_{i=1}^9 [F_i] = 3h-f$ in $H^2(X,\Z)$, the image of the period map (\ref{three}) lands in $$
\{(e_1,\dots,e_9)\in E^9\,\big{|}\,e_1+\cdots +e_9=0\}/\mathfrak S_9=A_8\otimes E/W(A_8)\simeq \mathbb{P}^8.$$
The last isomorphism follows from a well-known theorem of Looijenga
\cite{looijenga1976root}. Applying the definition of $\psi_{S_0}$
gives a very explicit construction of (\ref{three}):

\begin{definition} Fix a smooth cubic $D\subset \mathbb{P}^2$. Define
$E:={\rm Pic}^0(D)$ and let $A\colon {\rm Sym}^9E\to E$ denote the addition map. For a generic $\gamma\in PGL_3(\C)$, set $D\cap \gamma(D)=\{p_i\}_{i=1}^9$ and $q_i:=\gamma^{-1}(p_i)\in D$. We define \begin{align}\label{four}
\begin{aligned} \Psi\colon PGL_3(\C) & \dashrightarrow A^{-1}(0)\simeq \mathbb{P}^8 \\
\gamma &\mapsto  \{\cO_D(p_i- q_i)\}_{i=1}^9.\end{aligned}\end{align}
 \end{definition}

\begin{theorem}\label{surj} The rational map $\Psi$ from (\ref{four}) is dominant. 
Thus, the period mapping for Type II$_b$ surfaces is dominant. \end{theorem}

\begin{proof} Let $G\subset PGL_3(\C)$ be the finite subgroup for which $g(D)=D$.
We claim that $\Psi$ extends, as a morphism, from $U$ to $PGL_3(\C)\setminus G$.
This is easy: the map $\Psi$ extends continuously because $D\cap \gamma(D)$ is still a finite
set for all $\gamma\in PGL_3(\C)\setminus G$. Normality of
$PGL_3(\C)\setminus G$ implies that a continuous extension is algebraic.

\begin{figure}
\includegraphics[width=2.6in]{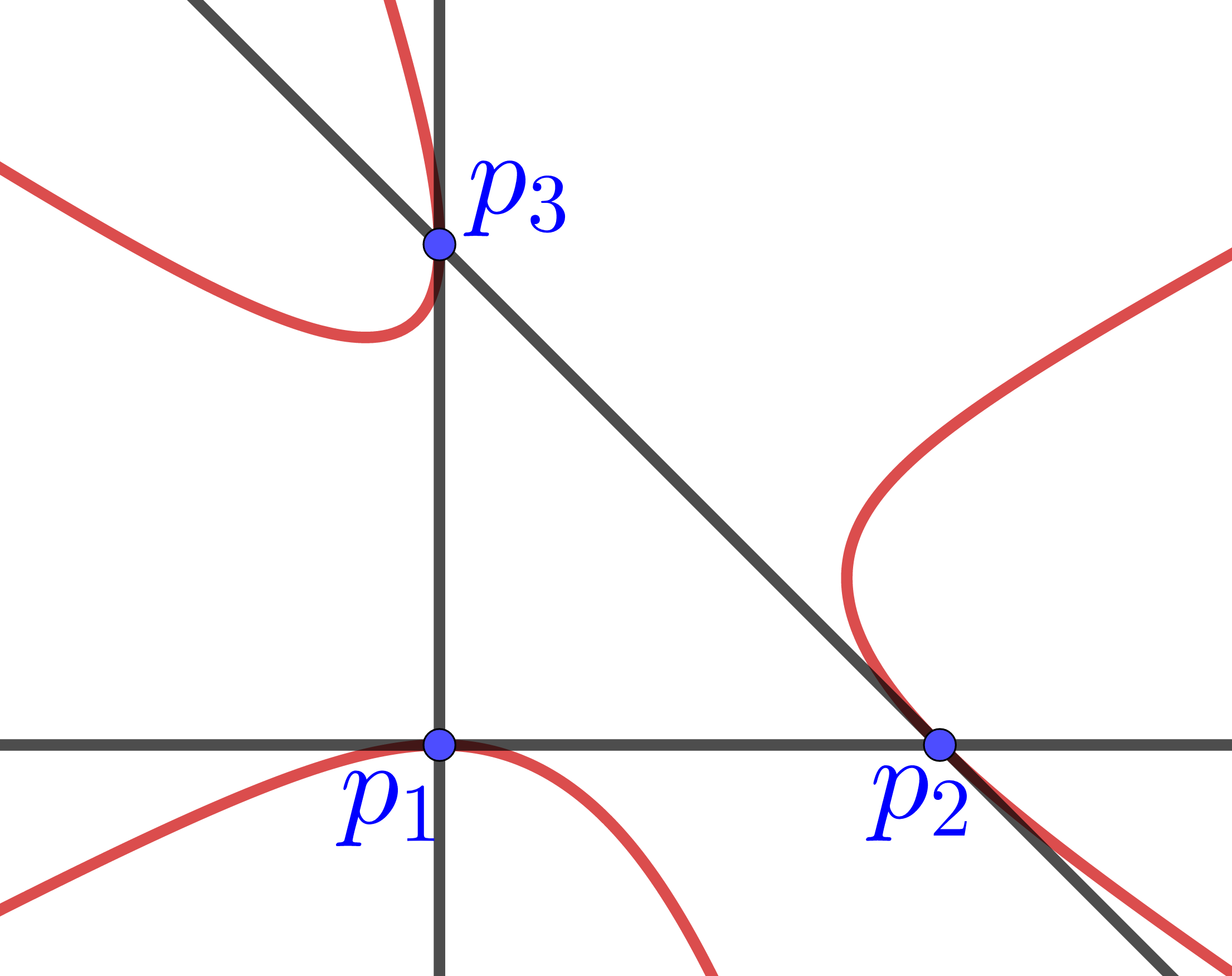}
\caption{The pencil generated by two cubics, shown in red and black,
with set-theoretic base locus three blue points.}
\label{cubics}
\end{figure} 

We choose $D$ and $\gamma$ carefully so that the set $D\cap \gamma(D)$
has only three elements. Concretely, consider the extremal
cubic pencil $X_{9111}\to \mathbb{P}^1_{[\lambda:\mu]}$ in the notation of
\cite{miranda1986on-extremal-rational},
given by the equation $$\lambda(x^2y+y^2z+z^2x)+\mu(xyz)=0,$$
see Figure \ref{cubics}.
Let $D:=D_{[\lambda:\mu]}$ be a generic fiber, and let $\gamma = {\rm diag}(1,\zeta_3,\zeta_3^2)$
where $\zeta_3$ is a primitive third root of unity. Then $\gamma(D)=D_{[\zeta_3\lambda:\mu]}$
and so $D$ and $\gamma(D)$ generate the pencil. The intersection multiset $D\cap \gamma(D)$
 is $\{3p_1,3p_2,3 p_3\}$ where $$p_1=[1:0:0],\,p_2=[0:1:0],\,
p_3=[0:0:1].$$ Since this $\gamma\in PGL_3(\C)$ fixes $p_1$, $p_2$, $p_3$,
the period $\Psi(\gamma)=\{0,\dots,0\}\in {\rm Sym}^9E$ vanishes.
To prove that $\Psi$ is dominant, it suffices to show that there is no small deformation
$\gamma'\in PGL_3(\C)$ of $\gamma$ for which $\Psi(\gamma')=\{0,\dots,0\}$.

Suppose, to the contrary, that there were. Since $\Psi(\gamma')=\{0,\dots,0\}$,
every base point in $D\cap \gamma'(D)$ is fixed by $\gamma'$. If $|D\cap \gamma'(D)|\geq 4$,
then $\gamma'$ must fix a line in $\mathbb{P}^2$. This is impossible for a small deformation
of $\gamma$, which has isolated fixed points. Conversely, $|D\cap \gamma'(D)|\geq 3$
because each of $p_1$, $p_2$, $p_3$ deforms to some fixed point of $\gamma'$.
Hence, $\gamma'$ fixes exactly three points $p_1'$, $p_2'$, $p_3'$. Furthermore,
$D\cap \gamma'(D)=\{3p_1',3p_2',3p_3'\}$ as a multiset, again because $\gamma'$ is near $\gamma$, and the map
$$PGL_3(\C)\setminus G \to \mathrm{Sym}^9(D)$$
sending $\gamma' \mapsto D\cap \gamma'(D)$ with multiplicities is continuous.

Since
${\mult}_{p_i'}(D\cap \gamma'(D))\geq 2$ we deduce that $\gamma'$ preserves
the tangent direction $T_{p_i'}D$ and the corresponding tangent line $L_i'$. Thus,
$\gamma'\in PGL_3(\C)$ fixes the point $L_i'\cap L_j'\in \mathbb{P}^2$. But, as we noted
before, $\gamma'$ only fixes three points (this holds not just on $D$ but
in the ambient plane $\mathbb{P}^2$). Using that $\gamma'$ is a small deformation
of $\gamma$, we deduce that $$L_1'\cap L_2'=p_2',\,L_2'\cap L_3'=p_3',\,L_3'\cap L_1'=p_1'.$$
 
 Write $p_i'=p_i+t_i$ for a translation $t_i$. By the addition law on a cubic,
 we have $$2p_1'=-p_2',\,2p_2'=-p_3',\,2p_3'=-p_1'$$ from which we can conclude
 that $t_1=(-2)^3t_1$ i.e. $t_1$ is $9$-torsion. But since $t_i$ are small, we conclude
 that $t_1=t_2=t_3=0$ and so $p_i'=p_i$.
 
 Thus, $\gamma'$ fixes $(p_1,p_2,p_3)$, implying that $\gamma'\in (\C^*)^2\subset PGL_3(
 \C)$ lies in the maximal torus associated to the coordinates $[x:y:z]$. Furthermore, $\gamma'$
 preserves the base locus scheme $D\cap \gamma'(D)$, as this is the unique subscheme of $D$
 which has length $3$ at each of $p_1,p_2,p_3$. So $\gamma'$ induces an automorphism of
 the pencil generated by $D$ and $\gamma'(D)$. Since the automorphism group of a rational
 elliptic surface is discrete, and $\gamma'$ is a small deformation of $\gamma$, 
 the automorphism $\gamma'$ must have order $3$. But no nontrivial
 small deformation of $\gamma={\rm diag}(1,\zeta_3,\zeta_3^2)$
 {\it within the torus} $(\C^*)^2$ has order $3$. This is a contradiction.
\end{proof}

\begin{remark} Our original proof of Theorem \ref{surj} checked by computer
that $d\Psi$ was non-degenerate for an explicitly chosen $D$ and $\gamma$.
\end{remark}

\begin{proof}[Proof of Theorem \ref{thm:main}]
By the discussion at the end of Section \ref{sec:period-map}, $P$ is dominant
if $P^{\rm II}$ is. The latter follows from Theorem \ref{surj}.\end{proof}

\section{Type II$_f$ degenerations}\label{sec:second-degen}

We consider in this section degenerations of $S\to C$ that keep the base $C$ constant. These are never of Type II$_b$ because in all such degenerations, $j(C)\to \infty$.

Take a one-parameter deformation of $a,b\in H^0(C,\cO_C(4p)), H^0(C,\cO_C(6p))$ over $(B,0)$ 
until the discriminant $4a_0^3+27b_0^2=0\in H^0(C,\cO_C(12p))$ vanishes identically.
For instance, we can take the fiber over $0\in B$
to be $$y^2=x^3-3r^2x+2r^3$$ with $r\in H^0(C,\cO_C(2p))$. The degeneration
$$\overline{\cS}\to C\times B\to (B,0)$$
of elliptic surfaces has a central fiber $\oS_0\to C$ whose generic fiber
is irreducible nodal, with two cuspidal fibers over the zeroes of $r$. In particular,
the normalization $\oS^\nu_0:=X\to C$ is the smooth $\mathbb{P}^1$-bundle
$X=\mathbb{P}_C(\cO\oplus L)$ and $\oS_0$ is reconstructed
from gluing a bisection $D$ of $X\to C$, branched over the two zeroes of $r$. This bisection $D$
is glued along the involution switching the two sheets of $\nu \colon D\to C$.

For future reference, note that $\NS(X)\simeq H^2(X,\Z)$ is spanned by the $\mathbb P^1$-fiber class $f$ and the class of the section $s_{\infty} = \mathbb P_C(\cO\oplus 0)$, with intersection form
$$f\cdot f = 0,\,\,\,\, s_\infty\cdot f = 1, \,\,\,\, s_\infty\cdot s_{\infty} = -1,$$
and $K_X = -f-2s_\infty$. The other natural section $s_0 = \mathbb P_C(0\oplus L)$ has class $f+s_\infty$.

The bisection $D\subset X$ has genus 2, being a double cover of $C$ branched over two points.
Thus its cohomology class is $[D]=2f+2s_\infty = -K_X+f = 2s_0$. Note that $[D]^2=4$ and $[D]\cdot K_X = -2$. The section $s$ that is present on the smooth surfaces in the family $\mathcal S$ limits to $s_\infty$, which is the unique section of $X$ disjoint from $D$.

\begin{proposition}\label{limits-ok} Generically, two singular fibers limit to each
cuspidal fiber of $\oS_0$. The limits of the remaining eight singular fibers
lie over a degree $8$ divisor in $C$. The only restriction on this divisor
is that it is linearly equivalent to $8p$. \end{proposition}

\begin{proof} Consider a deformation of the Weierstrass equation 
$$y^2=x^3-(3r^2+\epsilon g_4)x+(2r^3+\epsilon g_4r+\epsilon^2g_6)$$ where
$g_d\in H^0(C,\cO_C(dp))$ has degree $d$. The discriminant $\Delta = 4a^3+27b^2$
is $$\Delta = 9r^2(12rg_6 - g_4^2)\epsilon^2+\cO(\epsilon^3).$$
Thus, the Zariski closure of the discriminant divisor is
$$\lim_{\epsilon\to 0} {\rm div}(\Delta) =
2\cdot {\rm div}(r)+{\rm div}(12rg_6-g_4^2).$$ For fixed $r$, the
sections $rg_6$ form a linear subspace
$\mathbb{P}^5\subset \mathbb{P}^7=\mathbb{P}H^0(C,\cO(8p))$ 
of codimension $2$. The sections $g_4^2\in \mathbb{P}H^0(C,\cO(8p))$
are the image of the degree $2$ Veronese embedding, followed by a linear
projection $$v_2\colon  \mathbb{P}^3\hookrightarrow \mathbb{P}^9=
\mathbb{P}{\rm Sym}^2H^0(C,\cO(4p))\dashrightarrow \mathbb{P}^7.$$
The inverse image of $\{{\rm div}(rg_6)\}=\mathbb{P}^5\subset \mathbb{P}^7$ is a copy of
$\mathbb{P}^7\subset \mathbb{P}^9$ under the linear projection.
Thus, the vanishing loci of linear combinations are represented geometrically as the join
of the projective subvarieties $v_2(\mathbb{P}^3)$, $\mathbb{P}^7\subset \mathbb{P}^9$.
This join is all of $\mathbb{P}^9$. Thus, we can realize any divisor in
$|8p|$ as $\lim_{\epsilon\to 0}{\rm div}(\Delta)-2\cdot {\rm div}(r)$. 
\end{proof}

For general $g_4$ and $g_6$, the punctured family over $B\setminus 0$
has smooth total space. The threefold $\overline{\cS}$ is a double cover
branched over the vanishing locus of the cubic
$x^3-(3r^2+\epsilon g_4)x+(2r^3+\epsilon g_4r+\epsilon^2g_6)$,
so it can only be singular where two of the roots
of the cubic coincide. This shows that the singular locus
$\overline{\cS}_{\rm sing}\subset V(y,x-r,\epsilon)$ is contained
in the singularities of the fibers of $\oS_0\to C$.

Since $\epsilon^2\mid\mid\Delta$, the local equation of the double
cover is generically $y^2=u^2+\epsilon^2$ along the nodes of $\oS_0\to C$.
So the nodes form a family of $A_1$-singularities in $\overline{\cS}$. At the nodes on the fibers
lying over ${\rm div}(12rg_6-g_4^2)$, the local equation is rather $y^2=u^2+v\epsilon^2$.
Thus, to find a semistable model $\cS\to (B,0)$,
we simply blow up the double locus of $\oS_0$ in the total space $\overline{\cS}$.

The resulting central fiber is $S_0=X\cup_DV$
for a ruled surface $V\to C$, which
contains $D$ as a bisection and has $8$ reducible fibers over the points
in ${\rm div}(12rg_6-g_4^2)$,
see Figure \ref{typeIIf}. Thus
$V\sim Bl_{p_1,\dots,p_8}X$ is deformation-equivalent
to the blow-up of $X$ at $8$ points on $D$, with the double locus
on $V$ identified with $D$ via the strict transform. It is only 
deformation-equivalent because $V\to C$ could be the projectivization
of a non-split extension of $L$ by $\cO$. Regardless, we can identify
$$H^2(V,\Z)=H^2(X,\Z)\oplus_{i=1}^8\Z E_i$$
and $[D] = 2s_0 - [E_1]-\cdots-[E_8]=-K_V+f$. 

\begin{figure}
\includegraphics[width=5in]{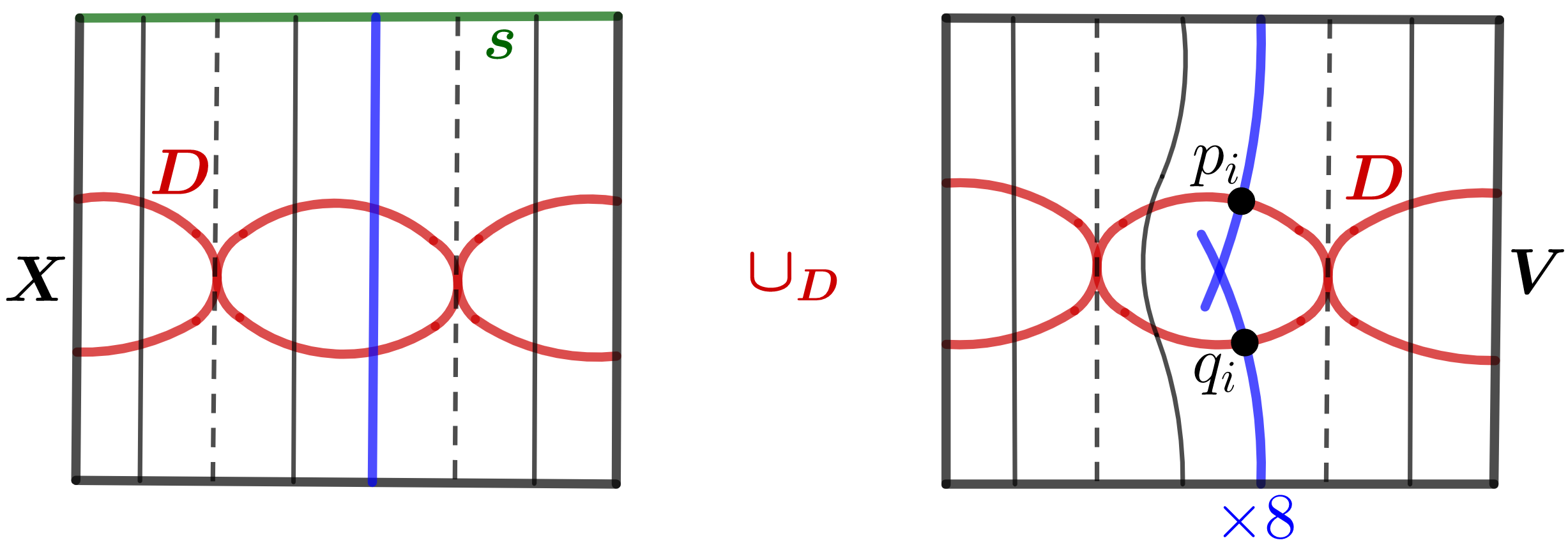}
\caption{A Type II$_f$ surface $S_0 = X\cup_D V$ with the genus $2$ double locus $D$ shown in red, the section $s$ in green, limits of $8$ nodal fibers in blue, and limits of pairs of nodal fibers dashed.}
\label{typeIIf}
\end{figure}

\begin{definition} We call the degeneration $\cS\to C\times B\to (B,0)$
a {\it Type II$_f$ degeneration} and we call the central fiber $S_0$
a {\it Type II$_f$ elliptic surface}. \end{definition}
%
%

From Section \ref{sec:mhs} and Proposition \ref{period-prop}, the mixed Hodge structure
of a Type II$_f$ surface has a period map to $E_8\otimes \cE$ which
can be described as follows. Consider the sublattice $\{K_V,f\}^\perp\subset H^2(V,\Z)$. This is isometric to the root lattice
$$D_8 = \{(a_1,\dots,a_8)\in \Z^8 \,\big{|}\, a_1+\cdots+a_8\in 2\Z \}$$
via the map $(a_1,\dots,a_8) \mapsto \sum_{i=1}^8 a_i[E_i]  - \left(\frac{1}{2} \sum_{i=1}^8 a_i\right) f$. When this isometry is understood, we will refer to $\{K_V,f\}^\perp$ simply as $D_8$.

 Let $E:={\rm Pic}^0(D)/{\rm Pic}^0(C)$ be
the Prym variety of the double cover $\nu\colon D\to C$.
We define a period homomorphism \begin{align} 
\label{iif_period} 
\begin{aligned} \psi_{S_0}\colon D_8& \to E 
\\
\gamma & \mapsto \mathcal{L}_\gamma\big{|}_D\textrm{ mod }{\rm Pic}^0(C)
\end{aligned}
\end{align}
by lifting an element $\gamma\in D_8$ to an element $\cL_\gamma\in {\rm Pic}(V)$.
These lifts form a ${\rm Pic}^0(C)$-torsor and thus the image of
$\cL_\gamma\big{|}_D\in {\rm Pic}^0(D)$ under the map to $E$ is well-defined.

\begin{remark} The period point $\psi_{S_0}\in{\rm Hom}(D_8,E)$ determines,
up to a finite isogeny, the period point in $E_8\otimes E$.
The extensions of an element of ${\rm Hom}(D_8, E)$ 
to an element of ${\rm Hom}(E_8, E)$ are a torsor over
${\rm Hom}(E_8/D_8,E)=E[2]$. \end{remark}

\begin{proof}[Proof of Theorem \ref{thm:main2}]
To show $\deg P>1$, it suffices to prove that the moduli of Type II$_f$
surfaces (appearing as limits of elliptic surfaces in $F$)
also dominate the boundary divisor $\Delta$.
This follows from Theorem \ref{surj2} below. \end{proof}

\begin{theorem}\label{surj2} The period mapping for Type II$_f$ surfaces is dominant.\end{theorem}

\begin{proof} The period point $\psi_{S_0}$
and limit mixed Hodge structure of $\cS$ are encoded,
up to a finite map, in the data $(\nu \colon D\to C,\{r_i\}_{i=1}^8)$
consisting of
\begin{enumerate}
\item a degree $2$ map $\nu\colon D\to C$ from a genus $2$ to a genus $1$ curve, and
\item a multiset of $8$ points $\{r_1,\dots,r_8\}\subset C$.
\end{enumerate}
Let $\iota\colon D\to D$ be the involution switching
the sheets of $\nu$ and let $\{p_i,q_i\}=\nu^{-1}(r_i)$. Then
$\cO_D(p_i-q_i)\in {\rm Pic}^0(D)$  gives, upon quotienting by 
${\rm Pic}^0(C)$, the period $$\psi_{S_0}(F_i-F_i')=[\cO_D(p_i-q_i)]\in E,$$ 
where $F_i+F_i'$  is a reducible fiber of the ruling $V\to C$. Ranging over the 
eight reducible fibers, the tuple \[(\cO_D(p_i-q_i)\textrm{ mod }{\rm Pic}^0(C))_{i=1}^8\in E^8\] encodes $\psi_{S_0}$ up to torsion, because $\bigoplus_{i=1}^8\Z(F_i-F_i')\subset D_8$ has
finite index.

Let $\{r_9, r_{10}\}\in C$ be
the branch points of $\nu$. Then $\nu$ is determined
by the monodromy representation $\rho\colon \pi_1(C\setminus \{r_9,r_{10}\},*)\to \Z_2$. Let ${\rm Prym}^2\cC$ be the moduli space of
Prym data $(C,\{r_9,r_{10}\},\rho)$ over the universal genus $1$ curve 
$\mathcal{C}\to \mathcal{M}_1$.
It is a Deligne-Mumford stack of dimension $2$,
one dimension for $j(C)$ and another for the element
$r_9-r_{10}\in {\rm Pic}^0(C)$,
well-defined up to sign. The data of $\rho$ is finite.

A point $r_i\in C$ determines $p_i$ up to switching
$p_i\leftrightarrow q_i$ which acts by negation on the image of $\cO_D(p_i-q_i)$ in $E$.
Thus, we globally get a well-defined map 
\begin{align}\begin{aligned} \Psi\colon {\rm Sym}^8\cC \times_{\mathcal M_1} 
{\rm Prym}^2\cC &\to \Z^8 \otimes \cE/\mathfrak S_8^{\pm} \\ 
(C, \{r_1,\dots,r_8\} ,\{r_9,r_{10}\},\rho)& 
\mapsto \{\cO_D(p_i-q_i)\textrm{ mod }{\rm Pic}^0(C)\}_{i=1}^8\end{aligned}\end{align}
where $\cE$ is the universal elliptic curve.
Since the image of each $\cO_D(p_i-q_i)$ in $E$ is only
well-defined up to sign, and the reducible fibers of $V\to C$ are unordered,
we must quotient the target by the signed permutation
group $\mathfrak S_8^{\pm}$.

Observe that ${\rm Sym}^8\cC \times_{\mathcal M_1} {\rm Prym}^2\cC$ is ten-dimensional.
There is a single condition ensuring that a point in the domain of $\Psi$ arises from a degeneration of surfaces 
in $F$: If $L\to C$ is the Hodge bundle, then $r_9+r_{10}\in |2L|$
and so by Proposition \ref{limits-ok}, $\{r_1,\dots,r_8\},\{r_9,r_{10}\}$ can arise so long as
$r_1+\cdots+r_8\in |8L|$ i.e.~the relation \begin{equation}\label{relation} r_1+\cdots+r_8- 4(r_9+r_{10})=0\in {\rm Pic}^0(C)
\end{equation} is satisfied.
So the Type II$_f$ limits of degenerations from $F$ are described by 
$$Z= \{\textrm{elements of }{\rm Sym}^8\cC \times_{\mathcal M_1} {\rm Prym}^2\cC\,\big{|}\,r_1+\cdots+r_8-4(r_9+r_{10})=0\}.$$
Our goal is to prove the dominance of the map 
$\Psi\big{|}_Z\colon Z\to \Z^8 \otimes \cE/\mathfrak S_8^{\pm}.$

Fix an elliptic curve fiber $E$ of $\mathcal{E}$,
consider the point $\{0,\dots,0\}\in {\rm Sym}^8E$, and
let $\ker_E(\Psi):= \Psi^{-1}(\{0,\dots,0\})$.
It suffices to prove that $Z\cap \ker_E(\Psi)$ contains,
as a component, some zero-dimensional scheme. Let
$L_E\subset {\rm Prym}^2\mathcal{C}$ be the sublocus
of Prym data whose Prym variety is $E$. It is a curve inside
the surface ${\rm Prym}^2\cC$. Then, ${\rm ker}_E(\Psi)$ contains,
as a component, an unramified double cover $M_E\to L_E$ on which
$r=r_1=\cdots = r_8$ and $r\in \{r_9,r_{10}\}$ because the morphism $D\to E$ sending $p\mapsto \cO_D(p-\iota(p))\textrm{ mod }{\rm Pic}^0(C)$ is surjective.

The defining equation (\ref{relation}) of $Z$ restricts
to $M_E$ to give the equation $$4(r_9-r_{10})=0\in {\rm Pic}^0(C),$$
i.e.~ $r_9-r_{10}\in {\rm Pic}^0(C)[4].$ The locus in $L_E$
on which $r_9-r_{10}$ is $4$-torsion is finite and 
non-empty. So the theorem follows.
\end{proof}

The proofs of Theorems \ref{surj} and \ref{surj2} 
suggest a rather wild conjecture:

\begin{conjecture}\label{fmconj} $F_{1,1}$ admits a period-preserving birational involution $S\leftrightarrow S'$
for which $j(C)=j(F')$ and $j(F)=j(C')$. Here $C, C'$ are the bases and $F, F'$ are the canonical fibers. 
Furthermore, $S$ and $S'$ are moduli spaces of stable vector bundles on each other of rank $2$,
determinant $\cO(s)$, and $c_2={\rm pt}$. A Fourier-Mukai transform induces an isomorphism
of their integral Hodge structures. \end{conjecture}

The existence of such a birational involution would give a geometric explanation for why
degenerations of Type II$_b$ and II$_f$ can have
the same periods, even though $j(C)\to \infty$ in the former, while
$j(F)\to \infty$ in the latter.

\section{A family losing dimension}

Let $F^{\rm cusp}\hookrightarrow F$ be the closure of the sublocus of elliptic fibrations
$S\to C$ which have six cuspidal (Kodaira type II) fibers. These fibrations are isotrivial and
have a Weierstrass form $y^2=x^3+b$ for some $b\in H^0(C,\cO_C(6p))$. There is a fiber preserving
automorphism $\sigma\colon S\to S$, given by $$\sigma\colon (x,y)\mapsto (\zeta_3x,-y)$$
and $\sigma^*\Omega_S=\zeta_6\Omega_S$ acts nontrivially on the holomorphic $2$-form by a
primitive sixth root of unity.
Furthermore, since $\sigma$ preserves $s$ and $f$, it defines an element $\sigma^*\in \Gamma=O(I\!I_{2,10})$ which is easily checked to fix only the origin of $I\!I_{2,10}$.
So $\sigma^*$ endows $I\!I_{2,10}$ with the structure of a Hermitian lattice of hyperbolic signature
$(1,5)$ over the Eisenstein integers $\Z[\zeta_6]$ and 
$$\bB:=\mathbb{P}\{x\in I\!I_{2,10}\otimes \C\,\big{|}\,x\cdot \overline{x}>0,\, \sigma^*x=\zeta_6x\}\subset \bD$$
is a Type I Hermitian symmetric subdomain (a complex ball), of dimension $5$.
Letting $\Gamma_0:=\{\gamma\in \Gamma\,\big{|}\,\gamma\circ \sigma^* = \sigma^*\circ \gamma\}$
be the group of Hermitian isometries, we get a period map to a
 $5$-dimensional ball quotient $$F^{\rm cusp}\to \bB/\Gamma_0.$$
 But $\dim F^{\rm cusp}=1+5=6$ with parameters
 corresponding to $j(C)$ and the relative locations of the six cuspidal fibers.
 Thus $P\big{|}_{F^{\rm cusp}}$ has positive fiber dimension.
 
 It seems likely
 that $P\big{|}_{F^{\rm cusp}}$ is surjective, with generic fiber dimension $1$. Regardless, this gives a second example, after Ikeda's \cite{ikeda2019bielliptic}, 
proving that $P$ is {\it not} a finite map,
even though it is generically finite by Theorem \ref{thm:main}:

\begin{corollary} $P$ is not finite.\end{corollary}

\section{Mixed Hodge Structures}\label{sec:mhs}

\subsection*{MHS of a normal crossings surface}\label{subsec:surf}

Let $S_0$ be a reduced normal crossings surface with smooth double locus and no triple points.
Our goal in this section is to explicitly describe the mixed Hodge structure on $H^2(S_0)$.
Let $S_0 = \bigcup_{i=1}^m S_i$ with the double curve $D_{ij} = S_i\cap S_j$
a smooth, possibly disconnected or empty curve for all $i<j$. Let $D:=\bigcup_{i<j} D_{ij}$.
The Mayer-Vietoris sequence associated to a covering of $S_0$ by neighborhoods of the irreducible components $S_i$ reads:
\begin{equation}\label{mayerv} \bigoplus_{i=1}^m H^1(S_i) \overset{\iota^*}\to 
\bigoplus_{i< j} H^1(D_{ij}) \to H^2(S_0) \to  \bigoplus_{i=1}^m H^2(S_i) 
\overset{{\rm res}}\longrightarrow \bigoplus_{i< j} H^2(D_{ij}).\end{equation}
Here $\iota^*$ and ${\rm res}$ are signed restriction maps.
Let $K\subset \bigoplus H^2(S_i)$ be the kernel of the morphism ${\rm res}$, that is
$K=\{(\alpha_i\in H^2(S_i))\,\big{|}\,\alpha_i\cdot D_{ij} = \alpha_j\cdot D_{ij}\}$.
Define
$$J:= \coker(\iota^*).$$
By exactness of the sequence (\ref{mayerv}), we obtain a short exact sequence
$$0 \to J \to H^2(S_0) \to K \to 0.$$ In fact, it is a short exact sequence
of mixed Hodge structures with left hand term $J$ pure of weight $1$, and the right 
hand term $K$ pure of weight $2$.

\begin{proposition}\label{hodge-ext}
If $p_g(S_i)=0$ for all components $S_i\subset S_0$ (equivalently,
$K$ is Hodge-Tate of weight $2$)
then the Carlson classifying map \cite{carlson1985one}
$$\phi:K \to {\rm Jac}(J)$$ of the extension
coincides with the Abel-Jacobi map. More precisely, an element of $K$ is a tuple $(\alpha_i\in H^2(S_i,\Z))$
represented by line bundles $\mathcal{L}_i$ such that for each $i<j$ we have
$c_1(\mathcal{L}_i|_{D_{ij}}) -c_1(\mathcal{L}_j|_{D_{ij}}) =0\in H^2(D_{ij})$. Then $\phi = \pi\circ {\rm AJ}
\circ \psi $ where
\begin{align*} 
(\alpha_i\in H^2(S_i,\Z))&\overset{\psi}\mapsto \textstyle \bigoplus_{i<j}
\mathcal{L}_i|_{D_{ij}} \otimes \mathcal{L}_j|_{D_{ij}}^{-1}\in {\rm Pic}^0(D), \end{align*}
${\rm AJ}\colon {\rm Pic}^0(D) \to  {\rm Jac}(D)=  {\rm Jac}(H^1(D))$ is the classical Abel-Jacobi isomorphism, and
$\pi \colon {\rm Jac}(D) \to {\rm Jac}(J)$ is the projection map.
\end{proposition}


\begin{proof}
Following Carlson's construction, the classifying map $\phi$ for a weight separated extension 
of mixed Hodge structures is given by the composition of two splittings. First, choose a left-splitting 
$a: H^2(S_0) \to J$ over $\Z$. Next, choose a right-splitting $b:K \to F^1 H^2(S_0)_\C$ over $\C$, 
which respects the Hodge filtration. The composition $a_\C\circ b: K \to J_\C$ gives the classifying map 
after passing to the Jacobian quotient:
$$\phi: K \to J_\C / (J_\Z + F^1 J_\C).$$
For $a$, it suffices to produce a morphism on homology $\ker(\iota_*) \to H_2(S_0)$, and then use the
universal coefficient theorem to give a map in the opposite direction:
$$H^2(S_0) \to H_2(S_0)^* \to \ker(i_*)^* \simeq \coker(\iota^*) = J.$$

To define the morphism $\ker(\iota_*) \to H_2(S_0)$, choose a basis for $\ker(\iota_*)$ at the singular 
chain level: tuples of 1-cycles $t_k=(\gamma^k_{ij}\in \mathcal{Z}_1(D_{ij}))$ such that for each $i$,
$$\sum_j \iota_*(\gamma^k_{ij}) = \partial(\Gamma^{k}_i) \textrm{ for some  }\Gamma_i^k\in \mathcal{C}_2(S_i).$$
We use the convention that $\gamma_{ij} = - \gamma_{ji}$. Choosing such $\Gamma_i^k$ for each
$t_k$ in the basis of $\ker(i_*)$, we construct a $2$-cycle, see Figure \ref{capping}, 
$$T_k=\bigcup_i \Gamma_i^{k} \in \mathcal{Z}_2(S_0).$$
We take the $1$-cycles $\gamma_{ij}^k$ to be $\Z$-linear combinations of some fixed
$2g(D_{ij})$ loops on each $D_{ij}$, whose union we call $\gamma$, chosen so that their
complement in $D_{ij}$ is a contractible $4g$-gon.
The assignment $t_k \mapsto [T_k]\in H_2(S_0)$ then induces a splitting 
$$a: H^2(S_0) \to J.$$

\begin{figure}
\includegraphics[width=4in]{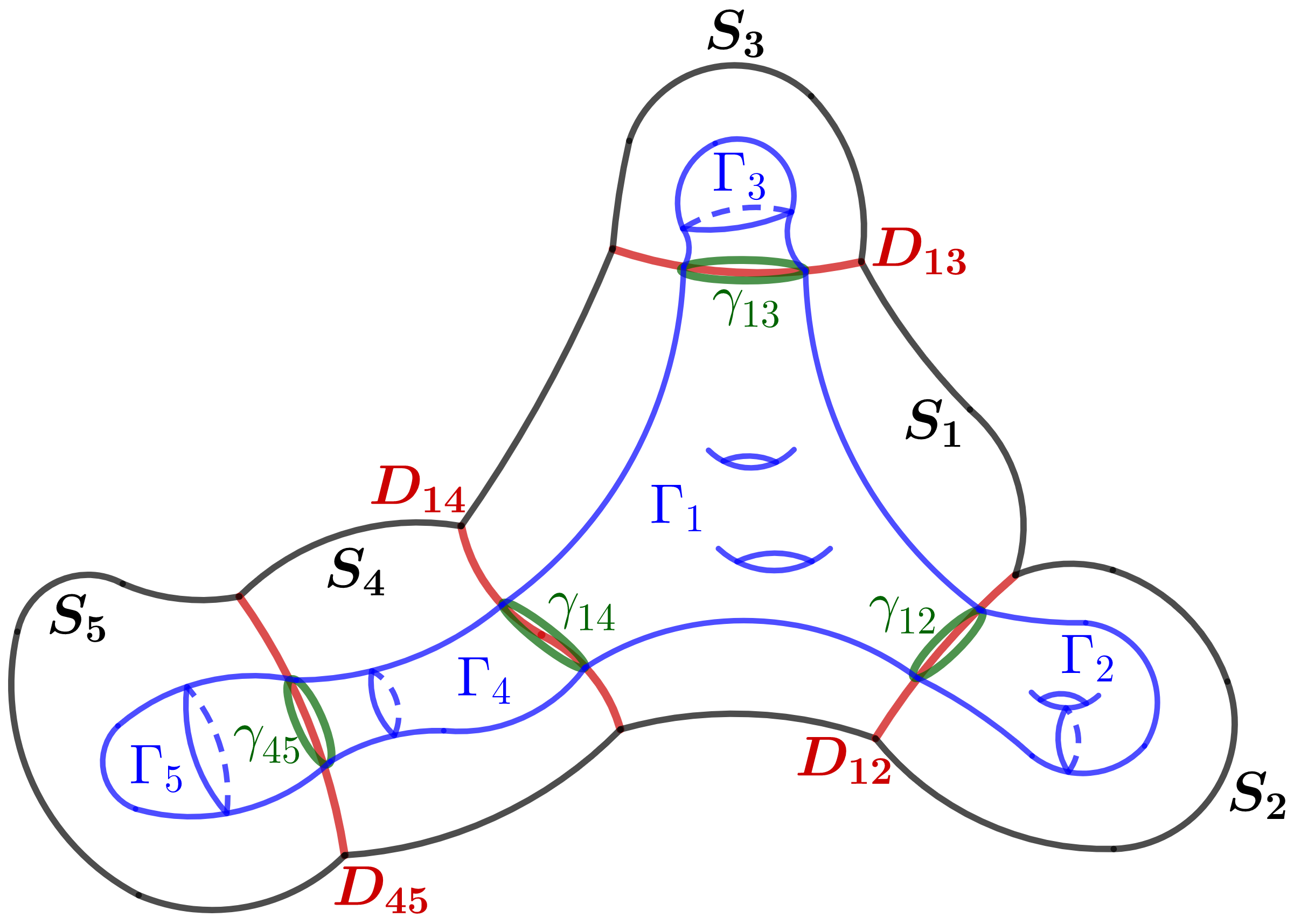}
\caption{Heuristic diagram of irreducible components $S_i$ in black, double curves $D_{ij}$ in red, $1$-cycles $\gamma_{ij}\subset D_{ij}$ in green, and
$2$-cycles $\Gamma_i\subset S_i$ capping the $1$-cycles in blue.}
\label{capping}
\end{figure}

To construct a splitting $b$, we use the \v{C}ech-de Rham model of $H^2(S_0,\C)$, and its Hodge filtration $F^1$. 
An element of $H^2(S_0,\C)$ is represented by two tuples of differential forms:
$$(\omega_i\in \cZ^2(S_i))_i\textrm{ and } (\theta_{ij}\in \mathcal{A}^1(D_{ij}))_{i<j}$$
such that for all $i<j$, we have $\omega_i|_{D_{ij}} - \omega_j|_{D_{ij}} = d\theta_{ij}$.
If furthermore $\theta_{ij}\in \mathcal{A}^{1,0}(D_{ij})$ for all $i<j$, then the element lies in $F^1H^2(S_0,\C)$.

Given $(\alpha_i)\in K = \ker({\rm res})$, we know that $\alpha_i|_{D_{ij}} - \alpha_j|_{D_{ij}}=0\in H^2(D_{ij})$. 
To define $b:K \to F^1 H^2(S_0,\C)$, select a basis for $K$; for each basis element $(\alpha_i)\in K$ there 
exists line bundles $\mathcal{L}_i$ such that $c_1(\mathcal{L}_i)=\alpha_i$. Since each $S_i$ is projective, 
we may assume that the $\mathcal{L}_i \simeq \mathcal{O}_{S_i}(C_i - C'_i)$ where $C_i$ and $C_i'$ are ample 
effective curves on $S_i$ meeting each $D_{ij}$ transversely away from $\gamma$. We take
 $\omega_i\in \cZ^2(S_i)$ representing $c_1(\mathcal{L}_i)$ and supported on a small 
 neighborhood of $C_i\cup C'_i$. Since $\omega_i|_{D_{ij}} - \omega_j|_{D_{ij}}\in \cZ^2(D_{ij})$ 
 integrates to 0, it has a $\overline{\partial}$-primitive $\theta_{ij}\in \mathcal{A}^{1,0}(D_{ij})$,
 unique up to the addition of a holomorphic one-form.

To interpret the composition $\phi=a_\C \circ b:K \to J_\C$, we will regard $J_\C$
as $\Hom(\ker(\iota_*),\C)$. Then
$(a_\C \circ b)(\alpha_i)$ is the unique homomorphism $\ker(\iota_*)\to \C$ which sends $t_k$ to 
\begin{equation}\label{integralsum}
\sum_{i=1}^m \int_{\Gamma_i^k} \omega_i + \sum_{i<j} \int_{\gamma_{ij}^k} \theta_{ij}.
\end{equation}
We henceforth drop the index $k$ as we will consider a single basis vector
$t=t_k$.

We will make two simplifications in order to compare $\phi$ with the Abel-Jacobi map.
First, the chains $\Gamma_i$ can be replaced with $\Gamma_i+x_i$ for any
$x_i\in \mathcal{Z}_2(S_i)$ such that the tuple of homology classes $(x_i)$ is 
Poincar\'{e} dual to an element of $K$. By Lefschetz duality, there is a perfect pairing associated to the 4-manifold with boundary:
$$I:H_2(S_i - N_\epsilon(\gamma))\times H_2(S_i - N_\epsilon(\gamma), \partial) \to \Z,$$
and we have $\int_{\Gamma_i} \omega_i = I(C_i - C_i', \Gamma_i) \in \Z$. Since $(\alpha_i)$ is primitive in $K$, one can find $x\in K$ such that
$$I(C_i - C_i',x) = -I(C_i - C_i', \Gamma_i).$$
So replacing $\Gamma_i$ with $\Gamma_i+x_i$, we may assume that the first 
sum in (\ref{integralsum}) vanishes.

Second, the primitives $\theta_{ij}$ are not closed, so the 
second integral does not make sense on the homology classes $[\gamma^k_{ij}]$. 
To remedy this, we construct smooth 1-forms $\lambda_{ij}\in \cZ^1(D_{ij})$ supported 
away from $\gamma$ such that $d(\theta_{ij}+\lambda_{ij})=0$. Let $\ell_{ij}$ be a 
smooth 1-chain on $D_{ij}\setminus \gamma$ with boundary the signed intersection points:
$$\partial \ell_{ij} = (C_i - C_i')\cap D_{ij} - (C_j - C_j')\cap D_{ij}.$$
By Lemma \ref{ddbarlog} below, we may produce a form $\lambda_{ij}$ supported in a
neighborhood of $\ell_{ij}$. This allows us to write the Carlson map for our extension as:
$$\phi((\alpha_i)) = \left[ t \mapsto \sum_{i<j} \int_{\lambda_{ij}} (\theta_{ij}+\lambda_{ij}) \right] \in  J_\C / (J_\Z + F^1 J_\C).$$
But for any $\tau\in \Omega^1(D_{ij})$, since $\theta_{ij}\in \mathcal{A}^{1,0}(D_{ij})$ we have, again by Lemma \ref{ddbarlog},
$$\int_{D_{ij}} (\theta_{ij}+\lambda_{ij})\wedge \tau = \int_{D_{ij}} \lambda_{ij}\wedge \tau = \int_{\ell_{ij}} \tau.$$
Observe that the classical Abel-Jacobi map ${\rm AJ}\colon {\rm Pic}^0(D)\to {\rm Jac}(D)$ 
indeed sends $[\partial \ell_{ij}] \mapsto \int_{\ell_{ij}}$. The proposition follows. \end{proof}

Now, we produce the one-form $\lambda_{ij}$ with the desired properties.

\begin{lemma}\label{ddbarlog} Let $C$ be a Riemann surface and let $\mathcal{L}=\mathcal{O}_C(q-p)$. There is a
hermitian metric $h$ on $\mathcal{L}$, a $(1,0)$-form $\theta\in \mathcal{A}^{1,0}(C)$,
and a smooth $1$-form $\lambda$ supported in a neighborhood of a path
$\ell$ from $p$ to $q$ for which:
\begin{enumerate}
\item $\overline{\partial} \theta = \frac{i}{2\pi}\partial \overline{\partial}\log (h)$, 
\item $d\lambda = -\overline{\partial} \theta$, and
\item $\int \lambda\wedge \tau = \int_{\ell} \tau$ for any holomorphic one-form $\tau$.
\end{enumerate}
\end{lemma}

\begin{proof} Let $z$ be a chart to $\C$ from a neighborhood of $\ell$.
There exists a function $f\colon C\setminus \{p,q\}\to \C^*$ of the following form: 
$$f = \threepartdef{\frac{z-q}{z-p}}{z\in N_{\epsilon/2}(\ell)}{\textrm{smooth interpolation}}
{z\in N_{\epsilon/2}(\ell)^c\cap N_\epsilon(\ell)}{1}{z\notin N_\epsilon(\ell).}$$ Such a smooth interpolation exists
because $\frac{z-q}{z-p}$ has winding number zero along the boundary of $N_{\epsilon/2}(\ell)$.
Let $s\in {\rm Mero}(C, \mathcal{L})$
be a meromorphic section with a zero at $q$ and a pole at $p$. Then, there is a hermitian metric $h$
on $\mathcal{L}$ for which $h(s,\overline{s})=|f|^2$. The associated curvature form is
 $\tfrac{i}{2\pi}\partial \overline{\partial} \log |f|^2$ and since $c_1(\mathcal{L})=0$, we can find a $(1,0)$-form $\theta$ satisfying (1).
Furthermore, $\lambda = -\frac{i}{2\pi}(\overline{\partial} \log(f)-\partial \log(\overline{f}))$ is a $(0,1)$-form, supported in
$A:=N_{\epsilon/2}(\ell)^c\cap N_\epsilon(\ell)$ and satisfying (2).

It remains to check (3). We may write $\tau = dg$ for some holomorphic function 
$g\colon N_\epsilon(\ell)\to \C$. Applying Stokes' formula and the residue formula, we have
\begin{align*}\int_C \lambda\wedge \tau &=- \tfrac{i}{2\pi} \int_A\overline{\partial}\log(f) \wedge dg
= \tfrac{i}{2\pi} \int_A d(ig\cdot d\log(f))
= \tfrac{i}{2\pi}\int_{\partial A} ig\cdot d\log(f)  \\ 
&= -\tfrac{i}{2\pi}\int_{\partial N_{\epsilon/2}(\ell)}g \cdot d\log(\tfrac{z-q}{z-p}) 
= -\tfrac{i}{2\pi} (2\pi i)(g(q)-g(p)) = \int_\ell \tau. \end{align*}
 \end{proof}
 
 More generally, the lemma holds for any degree zero line bundle $\mathcal{O}_C(\sum (q_i-p_i))$,
 for a union of paths connecting each pair of points $p_i$ to $q_i$ by taking
the product of the hermitian metrics, and sum of the corresponding $\theta$'s and $\lambda$'s.

\begin{remark} To apply Lemma \ref{ddbarlog} to the proof of Proposition \ref{hodge-ext}, 
our forms $\omega_i$ must be such that $\omega_i|_{D_{ij}}-\omega_j|_{D_{ij}}$ is the
two-form $\frac{i}{2\pi} \partial\overline{\partial}\log(h)$ supported in a neighborhood of $\ell_{ij}$.
This is achieved by choosing 
$\omega_i = \frac{i}{2\pi} \partial \overline{\partial} \log(h_i)$ for hermitian metrics on $h_i$
on $\mathcal{L}_i$ (and similarly for $j$) so that $h=h_i/h_j$ is the desired hermitian metric on
$\mathcal{L}_i|_{D_{ij}}\otimes\mathcal{L}_j|_{D_{ij}}^{-1}$. Note though that we must allow
the two-form $\omega_i$ to be supported in a tubular neighborhood of $C_i\cup C_i'\cup \ell_{ij}$
rather than just $C_i\cup C_i'$. Since $\ell_{ij}$ is disjoint from $\gamma$, the argument
of Lemma \ref{hodge-ext} is unaffected. \end{remark}

\subsection*{Clemens-Schmid sequence}

Let $\mathcal{S}\to (B,0)$ be a degeneration of projective surfaces with smooth
total space, and reduced normal crossings central fiber $S_0=\bigcup_{i=1}^m S_i$ with smooth double locus. 
Assume furthermore that $p_g(S_i)=0$ for all $i$.

The monodromy
is unipotent by Clemens \cite{clemens1969picard}.
So let $N$ be the nilpotent logarithm of the
monodromy operator on $H^*(S_t)$.
We have the Clemens-Schmid sequence \cite{morrison1984clemens}
relating the integral cohomology of
$S_0$ and $S_t$:
\begin{equation}\label{cssequence}
0 \to H^0(S_t) \overset{N}\longrightarrow H^0(S_t) \to H_4(S_0) \to H^2(S_0) \to H^2(S_t) \overset{N}\longrightarrow H^2(S_t) .
\end{equation}
Since the monodromy operator acts trivially on $H^0(S_t)$, the first nilpotent operator in
(\ref{cssequence}) is identically $0$. 
Using these two observations, the Clemens-Schmid sequence can be shortened to:
\begin{equation}\label{cssequence2}
0 \to H^0(S_t) \to H_4(S_0)\simeq \Z^m \to H^2(S_0) \to H^2(S_t) \overset{N} \longrightarrow H^2(S_t) .
\end{equation}

The limit mixed Hodge structure $H^2(S_t)$ has a monodromy-weight filtration defined in terms
of $N$:  $\{0\} = W_0 \subset W_1 \subset W_2 \subset W_3 = H^2(S_t)$.
\begin{align*}
W_1 H^2(S_t) &= \im(N) ;\\
W_2 H^2(S_t) &= \ker(N) ;\\
W_3 H^2(S_t) &= H^2(S_t).
\end{align*}
We call $\ker(N)$ the {\it 1-truncated mixed Hodge structure}. To describe the 1-truncation explicitly, we
combine (\ref{cssequence2}) and (\ref{mayerv}) above at their common term
$H^2(S_0)$, with Mayer-Vietoris written horizontally
and Clemens-Schmid written vertically.

$$\begin{tikzcd}
 &   &  \im H_4(S_0) \arrow[d] \arrow[r] & \textrm{span}\{\xi_k\}   \arrow[d] &     \\
0 \arrow[r] & J  \arrow[r] \arrow[d] & H^2(S_0)  \arrow[r]   \arrow[d] & K \arrow[d] \arrow[r]  & 0 \\ 
0 \arrow[r] &  J \arrow[r] &  \ker(N) \arrow[r] & \Lambda \arrow[r] & 0.
\end{tikzcd}$$

Here, $\xi_k:= \sum_j [D_{jk}]-[D_{kj}]$, where $[D_{jk}]\in H^2(S_j)$ and $[D_{kj}]\in H^2(S_k)$ are the fundamental
classes of the double loci, and $\Lambda$ is the cokernel of $J\to \ker(N)$. We have that $\xi_k= c_1(\mathcal{O}_{\mathcal{S}}(S_k)|_{S_0})$. 
By Proposition \ref{hodge-ext}, we have $\xi_k\in \ker(\phi\colon K\to {\rm Jac}(J))$ because the line bundles
$\mathcal{O}_{\mathcal{S}}(S_k)\big{|}_{S_i}\simeq \mathcal{O}_{\mathcal{S}}(S_k)\big{|}_{S_j}$ agree
on the double locus. Hence, the Carlson extension homomorphism $\phi$ descends to a homomorphism
$$\psi_{S_0}\colon \Lambda\to {\rm Jac}(J)$$ encoding the 1-truncated mixed Hodge structure.

\subsection*{Application}

In this section, we apply the general results above to the mixed Hodge structures
associated to the degenerations of Type II$_b$ and II$_f$,
and relate their associated periods to the boundary of the
toroidal extension $(\bD/\Gamma)^{\rm II}$.

It is convenient to make an order $2$ base change
and resolution to the Type II$_b$ degenerations. The effect is to normalize the first component,
and insert a second component isomorphic to $\mathbb{P}^1\times E$ where
$E$ is the fiber over the node of $C_0$. This second
component is glued to the rational elliptic surface $X\to \mathbb{P}^1$ along the two fibers
$X_p, X_q$.

After the base change and resolution, we have that in both II$_b$ and II$_f$ degenerations,
the central fiber $S_0$ has two irreducible components and reduced normal crossings:
$S_0 = S_1 \cup_D S_2$. The double locus $D$ is a disjoint union of two copies
of the same elliptic curve $E$ in Type II$_b$ and a connected, smooth genus $2$
curve in Type II$_f$. Let $D_1\subset S_1$ and
$D_2\subset S_2$ denote the double locus restricted to each component.


In both cases, the divisor $D$ admits a natural involution $\iota$, and the image of the first
map $\iota^*$ in (\ref{mayerv})
is the $(+1)$-eigenspace of this involution on $H^1(D)$. The image of the restriction map ${\rm res}$ in 
(\ref{mayerv}) is a rank 1 subgroup of $H^2(D)\simeq H_0(D)$,
so the Mayer-Vietoris sequence takes the form:
\begin{equation}\label{mvsequence2}
0 \to H^1(D)^- \to H^2(S_0) \to H^2(S_1)\oplus H^2(S_2) \overset{\rm res}\longrightarrow \Z \to 0.
\end{equation}

{\bf Case II$_b$.} The component $S_1$ is a rational elliptic surface $X$,
with $D_1 = X_p \cup X_q$ a pair of isomorphic elliptic curve fibers.
The component $S_2$ is simply $\bP^1\times E$ with $D_2=\{0,\infty\}\times E$.
The involution on $D$ swaps the two isomorphic components. Note that since
$[X_p] = [X_q]\in H^2(S_1)$, and similarly for $S_2$, the two restriction maps
$H^2(S_i) \to H^2(D)\simeq H^2(E)^{\oplus 2}$ have the same image, namely
the diagonal. \vspace{5pt}

{\bf Case II$_f$.} The component $S_1$ is an elliptic ruled surface
$X\simeq \bP_C(\cO\oplus L)$, with $D_1$ a genus 2 bisection of class
$2s_0 = 2(s_\infty+f)$. The component $S_2$ is the blow up of (a deformation of)
$S_1$ at 8 points along $D_1$ with $D_2$ the proper transform of $D_1$ in the
blow up. The class of $D_2$ is $2s_0 - \sum e_i$. The involution on $D$ is induced
by the double cover map $\nu\colon D \to C$ which comes from the ruling of $X$. Since $D$ is
irreducible, $H^2(D)\simeq \Z$. \vspace{5pt}

In both cases, the Jacobian ${\rm Jac}(H^1(D)^-)=E$ is an elliptic curve.
In Type II$_b$ it is ${\rm Jac}(E)$ where $E$ is either of the double curves, while
in Type II$_f$ it is the Prym variety of the double cover map $\nu\colon D\to C$.
Thus, the mixed Hodge structure on $H^2(S_0)$ is encoded by a Carlson
extension map $\phi\in {\rm Hom}(K,E)$. By the previous subsection, 
this extension homomorphism descends to $\psi_{S_0}\in {\rm Hom}(\Lambda,E)$
where $$\Lambda = K/{\rm span}\{\xi_1,\xi_2\} = 
\ker(H^2(S_1)\oplus H^2(S_2)\overset{\rm res}\longrightarrow \Z)/\Z(D_1,-D_2).$$

%
%
%

There is a symmetric bilinear form on $H^2(S_0)$. Let
$$p\colon H^2(S_0)\to H^2(S_1)\oplus H^2(S_2)\xrightarrow{\rm PD}
H_2(S_1)\oplus H_2(S_2)\to H_2(S_0)$$
be restriction, followed by the Poincar\'e duality, followed by inclusion. Then
define $\alpha\cdot \beta := \langle \alpha, p(\beta)\rangle$ on $H^2(S_0)$. 
The map $H^2(S_0)\to H^2(S_t)$ respects the bilinear forms
on the source and target and the bilinear form descends to $K=\ker({\rm res})$.

By Poincar\'{e} duality and the Hodge index theorem, $H^2(S_1)\oplus H^2(S_2)$ is a
unimodular lattice of signature $(2,10)$, and it is odd since at least one summand contains $(-1)$-curves. Since $D_1^2+D_2^2 =0$, the lattice
vector $(D_1,-D_2)$ is isotropic, and its orthogonal complement
is precisely $\ker({\rm res})$. Hence the lattice $\Lambda$ is unimodular of signature $(1,9)$.

Our degenerating families are polarized by $\Z s\oplus \Z(s+f)\subset H^2(S_t)$. The monodromy operator
fixes these curve classes and hence we have a copy of $I_{1,1}\subset \ker(N)$. So $s$, $f$ extend over the
singular fiber by (\ref{cssequence}). They can be represented in $K$ as follows:
$(s,s)$, $(f,0)$ for Type II$_b$ and $(s_\infty, 0)$, $(f,f)$ for Type II$_f$, respectively.
In both cases, they span a sublattice of $\Lambda$ isometric to $I_{1,1}$
whose orthogonal complement we call $\Lambda_0\subset \Lambda$. We also have $\Lambda_0\simeq \Lambda/I_{1,1}$ canonically.
\begin{proposition}
The lattice $\Lambda_0$ is isometric to $E_8$ in both cases.
\end{proposition}

\begin{proof}
Note that $\Lambda_0$ is unimodular of signature $(0,8)$,
so it suffices to check that it is even. The orthogonal complement of $\{s,f\}$
in $\ker(N)$ is even because $f=K_{S_t}$ and $x\cdot x\equiv x\cdot K_{S_t}\textrm{ mod }2$
for any $x\in H^2(S_t)$. Hence, its image $\Lambda_0$ is even
because $\ker(N)\to \Lambda$ preserves the intersection form.
\end{proof}

\begin{remark}
The lattice $\Lambda_0$ can be described more directly using one irreducible component
(only up to finite index in the Type II$_f$ case). For Type II$_b$, the sublattice
$\{s,f\}^\perp \subset H^2(S_1)$ lies in $K$ and is even, unimodular of signature
$(0,8)$. So it maps isometrically to $\Lambda_0\simeq E_8$.
For II$_f$ the sublattice $\{D_2,f\}^\perp \subset H^2(S_2)$ lies in $K$ and
so maps isometrically to an index two sublattice $D_8\subset \Lambda_0\simeq E_8$.
\end{remark}

We summarize the results of this section in the following proposition:

\begin{proposition}\label{period-prop} Let $\mathcal{S}\to (B,0)$ be a degeneration of Type II$_b$ or Type II$_f$.
Let $K = \ker(H^2(S_1)\oplus H^2(S_2)\to H^2(D))$ be the kernel of signed restriction, and let
$\Lambda:=K/\Z(D_1,-D_2)$ and $\Lambda_0=\{s,f\}^\perp \subset \Lambda$. Let $E$ be 
 $\rm Pic^0$ of either double curve in Type II$_b$ and the Prym variety
 ${\rm Pic}^0(D)/{\rm Pic}^0(C)$ in Type II$_f$. 
 
The Carlson extension class $\phi\in {\rm Hom}(K,E)$ describing the mixed Hodge structure on
$S_0$ descends to  ${\rm Hom}(\Lambda,E)$, and so determines the $1$-truncated limit mixed
Hodge structure of the degeneration. This homomorphism further descends to
a period point $\psi_{S_0}\in {\rm Hom}(\Lambda_0,E)$ where $\Lambda_0\simeq E_8$. Explicitly:
\begin{enumerate}
\item[(II$_b$)] The period point $\psi_{S_0}$ given by the map sending
$\cL\in \{s,f\}^\perp\subset {\rm Pic}(S_1)$ to $\cL\big{|}_{X_p}\otimes \cL\big{|}_{X_q}^{-1}\in E$.
\smallskip
\item[(II$_f$)] The period point $\psi_{S_0}$ is determined up to $2$-torsion by the map
sending $c_1(\cL)\in \{D,f\}^\perp\subset H^2(S_2)$ to $\cL\big{|}_D \in {\rm Pic}^0(D)/{\rm Pic}^0(C)=E$.
\end{enumerate}
\end{proposition}

\section*{Appendix: Compact moduli}

KSBA theory \cite{kollar1988threefolds-and-deformations, alexeev1996moduli-spaces, kollar2023families}
gives a general method for constructing compact
moduli spaces of pairs $(X,B)$, consisting of a projective variety $X$
and a $\Q$-Weil divisor $B$, which form a so-called {\it stable slc pair}:
\begin{enumerate}
\item the pair $(X,B)$ has semi-log canonical singularities,
\item $K_X+B$ is $\Q$-Cartier and ample.
\end{enumerate}
In the case at hand, the pair $(\oS,\epsilon s)$ satisfies these conditions,
where $S\to \oS$ is the contraction to the Weierstrass form. The paper \cite{ascher2021moduli} of Ascher and Bejleri with an appendix by Inchiostro
 studies the corresponding
compactification by stable slc pairs $F\hookrightarrow \oF^W.$
Every degeneration with generic fiber in $F$ has a unique limit in $\oF^W$
called the {\it stable model}.

No information is lost when considering
Type II$_b$ degenerations because the stable
model $\oS_0$ uniquely determines $S_0$: It is the resolution of
ADE configurations in fibers. On the other hand, for Type II$_f$ degenerations,
most period information is lost: the stable model $\oS_0$ is the
gluing of $\bP_C(\cO\oplus L)$ along the bisection $D$.
Thus, the locus in $\oF^W$ corresponding to Type II$_f$ degenerations
has dimension $2$, remembering only the genus $2$ double cover $\nu\colon D\to C$.

To record more period information, we can instead choose a different
divisor on the general surface $S\in F$. Let $\textstyle R:=s+\sum_{i=1}^{12} f_i$
where $f_i$ are the singular fibers of $S\to C$,
counted with multiplicity. Because $(\oS,\epsilon R)$ is a stable slc pair, we may again
compactify the moduli space of such pairs using KSBA theory:
$F\hookrightarrow \oF^R$ where  $\oF^R$ is the closure
of the pairs $\{(\oS,\epsilon R)\,\big{|}\,S\in F\}$ in moduli of all stable slc
pairs.
Up to a finite map, $\oF^R$ remembers the period information of a Type II$_f$
degeneration (and this is still so for Type II$_b$ surfaces).

Thus, it is possible that the normalization of
$\oF^R$ actually dominates a toroidal compactification of $\bD/\Gamma$. An analogous result
for elliptic K3 surfaces $(g,d)=(0,2)$ holds by \cite{alexeev2020compactifications-moduli}.
We leave this as a conjecture:

\begin{conjecture} There is a morphism $(\oF^R)^\nu\to \overline{\bD/\Gamma}^\mathfrak{F}$
to some toroidal compactification, for an appropriately chosen fan $\mathfrak{F}$. \end{conjecture}

\bibliographystyle{amsalpha}
\bibliography{main}
\vspace{10pt}

\end{document}